\documentclass[12pt,a4paper,twoside]{article}

\usepackage{amsmath,amsfonts,amsthm,amsopn,color,amssymb}
\usepackage{palatino}
\usepackage{graphicx}
\usepackage[ps2pdf,colorlinks=true,urlcolor=blue,
citecolor=red,linkcolor=blue,linktocpage,pdfpagelabels,bookmarksnumbered,bookmarksopen]{hyperref}

\newcommand{\eps}{\varepsilon}

\newcommand{\R}{\mathbb{R}}
\newcommand{\Z}{\mathbb{Z}}

\newcommand{\RN}{{\mathbb{R}^N}}
\newcommand{\RT}{{\mathbb{R}^3}}

\renewcommand{\le}{\leqslant}
\renewcommand{\ge}{\geqslant}
\renewcommand{\a }{\alpha }

\renewcommand{\d }{\delta }

\newcommand{\g }{\gamma }

\renewcommand{\l }{\lambda}
\newcommand{\n }{\nabla }

\renewcommand{\t}{\theta}
\renewcommand{\O}{\Omega}

\newcommand{\G}{\Gamma}

\renewcommand{\H}{H^1(\RT)}
\newcommand{\Hr}{H^1_r(\RT)}

\newcommand{\Ne}{\mathcal{N}}

\renewcommand{\o}{\omega}

\newcommand{\D }{{\mathcal D}^{1,2}(\RT)}

\newcommand{\irt }{\int_{\RT}}

\def\bbm[#1]{\mbox{\boldmath $#1$}}

\newtheorem{theorem}{Theorem}[section]
\newtheorem{lemma}[theorem]{Lemma}

\newtheorem{remark}[theorem]{Remark}

\renewenvironment{proof}{\noindent{\textit{Proof.~}}}{$\hfill\square$\vspace{0.2 cm}\\}
\newenvironment{proofth}{\noindent{\textit{Proof of Theorem \ref{th}.~}}}{$\hfill\square$\vspace{0.2 cm}\\}
\newenvironment{proofmain}{\noindent{\textit{Proof of Theorem \ref{main}.~}}}{$\hfill\square$\vspace{0.2 cm}\\}

\title{Improved estimates and a limit case for the electrostatic Klein-Gordon-Maxwell system\footnote{The authors are supported by M.I.U.R. -
P.R.I.N. ``Metodi variazionali e topologici nello studio di
fenomeni non lineari''}}

\author{A. Azzollini \thanks{Dipartimento di Matematica ed Informatica, Universit\`a degli
Studi della Basilicata,  Via dell'Ateneo Lucano 10, I-85100
Potenza, Italy, e-mail: {\tt antonio.azzollini@unibas.it}}
 \; \& \;
L. Pisani\thanks{Dipartimento di Matematica, Universit\`a degli Studi di Bari, Via E. Orabona 4, I-70125 Bari, Italy, e-mail: {\tt pisani@dm.uniba.it}}
 \; \& \;
A. Pomponio\thanks{Dipartimento di Matematica, Politecnico di
Bari, Via E. Orabona 4, I-70125 Bari, Italy, e-mail: {\tt
a.pomponio@poliba.it}}}

\date{}

\begin{document}

\maketitle

\begin{abstract}
We study the class of nonlinear Klein-Gordon-Maxwell systems describing a standing wave (charged matter field) in equilibrium with a purely electrostatic field. We improve some previous existence results in the case of an homogeneous nonlinearity. Moreover, we deal with a limit case, namely when the frequency of the standing wave is equal to the mass of the charged field; this case shows analogous features of the well known \textquotedblleft zero mass case\textquotedblright for scalar field equations. 
\end{abstract}

\section{Introduction}

This paper is concerned with a class of Klein-Gordon-Maxwell systems written as follows
\begin{equation}    \label{main eq}
\left\{
\begin{array}{ll}
-\Delta u +[m^2-(e\phi-\o)^2 ]u-f'(u)=0 & \hbox{in }\RT
\\
\Delta \phi=e(e\phi-\o) u^2 & \hbox{in }\RT.
\end{array}
\right.
\end{equation}
This system was introduced in the pioneering work of Benci and Fortunato \cite{BF2} in 2002. It represents a standing wave $\psi=u(x)e^{i\omega t}$ (charged matter field) in equilibrium with a purely electrostatic field $\mathbf{E}=-\nabla \phi (x)$.
The constant $m\geq 0$ represents the mass of the charged field and $e$ is the coupling constant introduced in the minimal coupling rule \cite{F}.

It is immediately seen that \eqref{main eq} deserves some interest as system if and only if $e\ne0$ and $\o\ne 0$,
otherwise we get $\phi = 0$. Through the paper we are looking for \emph{nontrivial solutions}, that is solutions such that $\phi \ne 0$.

Moreover we point out that the sign of $\o$ is not relevant for the existence of solutions. Indeed if $(u,\phi)$ is a solution of \eqref{main eq} with a certain value of $\o$, then $(u,-\phi)$ is a solution corresponding to $-\o$. So, without loss of generality, we shall assume $\o>0$. Analogously the sign of $e$ is not relevant, so we assume $e>0$. Actually the results we are going to prove do not depend on the value of $e$.

Let us recall some previous results that led us to the present research.
The first results are concerned with an homogeneous nonlinearity
$f(t)=\frac{1}{p}|t|^p$.
Therefore \eqref{main eq} becomes
\begin{equation}    \label{eq}
\left\{
\begin{array}{ll}
-\Delta u +[m^2-(e\phi-\o)^2 ]u-|u|^{p-2}u=0 & \hbox{in }\RT
\\
\Delta \phi=e(e\phi-\o) u^2 & \hbox{in }\RT.
\end{array}
\right.
\end{equation}

As we said before, the first result is due to Benci and Fortunato \cite{BF2}. They showed the existence of infinitely many solutions whenever $p\in(4,6)$ and $0< \o < m$.

In 2004 D'Aprile and Mugnai published  two papers on this topic. In \cite{DM} they proved the existence of nontrivial solutions of (\ref{main eq}) when $p\in(2,4]$ and $\o$ varies in a certain range depending on $p$:
\[
0< \o < m \, g_0(p)
\]
where
\begin{equation*}
g_0(p)=\sqrt{\frac{p-2}{2}}.
\end{equation*}
Afterwards, in \cite{DM2}, the same authors showed that (\ref{main eq}) has no nontrivial solutions if $p\geq 6$ and $\omega\in (0,m]$ (or $p\leq 2$).

Our first result gives a little improvement on problem \eqref{main eq} with $p\in(2,4)$.
\begin{theorem}\label{th}
Let $p\in (2,4)$. Assume that $0< \o <  m\,g(p)$ where
\[
g(p)=\left\{
\begin{array}
[c]{ll}%
\sqrt{(p-2)(4-p)}\qquad & \text{if }2<p<3,\\
1 & \text{if }3\leq p<4,
\end{array}
\right.
\]
then \eqref{eq} admits a nontrivial weak solution $(u,\phi)\in \H\times \D$.
\end{theorem}

\begin{figure}[ht]
\centering
\includegraphics[height=6cm]{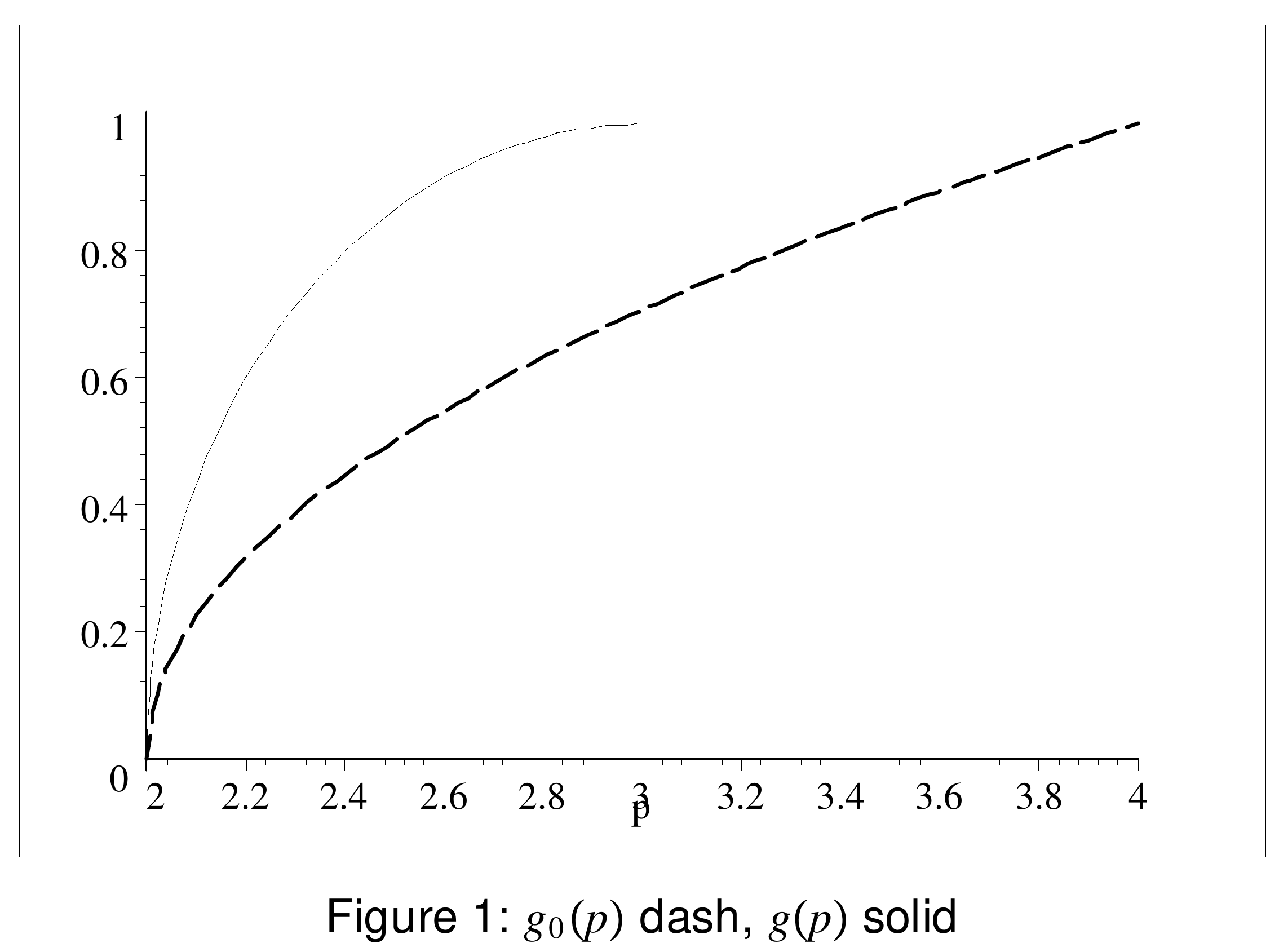}
\end{figure}

Under the above assumptions, the problem (\ref{eq}) is of a variational nature. Indeed its weak solutions $(u,\phi)\in \H\times \D$ can be characterized as critical points of the functional ${\cal S}:\H\times\D \to \R$ defined as
\[
{\cal S}(u,\phi)=\frac12 \irt |\n u|^2 -|\n \phi|^2+[m^2_0-(\o+e\phi)^2 ]u^2
-\frac{1}{p}\irt |u|^{p}.
\]

The first difficulty in dealing with the functional ${\cal S}$ is that it is strongly indefinite, namely it is unbounded both from below and from above on infinite dimensional subspaces.

To avoid this indefiniteness, we will use a well known reduction argument, stated in Theorem \ref{theo:red}. The finite energy solutions of \eqref{main eq} are pairs $(u,\phi_u)\in \H\times \D$, where $\phi_u\in \D$ is the unique solution of
\begin{equation} \label{seconda}
\Delta \phi=e(e\phi-\o) u^2 \quad \hbox{in }\RT
\end{equation}
(see Lemma \ref{le:phi}) and $u\in\H$ is a critical point of
\[
I(u)={\cal S}(u,\phi_u).
\]
The functional $I$ does not present anymore the strong indefiniteness. Under the assumptions of Theorem \ref{th}, it will be studied by using an indirect method developed by Struwe  \cite{S} and Jeanjean \cite{JJ}

In the second part of the paper we consider a more general nonlinearity $f(u)$.

Under usual assumptions, which describe behaviours analogous to $|t|^p$ (with $p\in(4,6)$), it is easy to get a generalization of the existence result (\cite{BF2}) of Benci and Fortunato; we shall state this generalization in Lemma \ref{main-1}. However we point out that all the quoted results share the assumption $\o < m$.

We are mainly interested to study the limit case $\o = m$, when \eqref{main eq} becomes
\begin{equation} \label{KGM}
\left\{
\begin{array}{ll}
-\Delta u +(2e\o\phi-e^2\phi^2 )u-f'(u)=0 & \hbox{in }\RT
\\
\Delta \phi=e(e\phi-\o) u^2 & \hbox{in }\RT.
\end{array}
\right.
\end{equation}
We notice that, in the first equation, besides the interaction term $(2e\o\phi-e^2\phi^2 )u$, there is no linear term in $u$. In this sense the situation described  by \eqref{KGM} is analogous to the \emph{zero-mass case} for nonlinear field equations (see e.g.~\cite{BL1}).

As in \cite{BL1}, in order to get solutions we need some stronger hypotheses on $f$, which force it to be inhomogeneous. More precisely we assume that $f:\R\to\R$ satisfies the following assumptions.
\begin{itemize}
\item[({\bf f1})] $f\in C^1(\R,\R)$;
\item[({\bf f2})] $\forall t \in\R\setminus\{0\}:\;\a f(t)\le f'(t)t$;
\item[({\bf f3})] $\forall t\in\R:$ $f(t)\ge C_1  \min(|t|^p,|t|^q)$;
\item[({\bf f4})] $\forall t\in\R:$ $|f'(t)|\le C_2 \min(|t|^{p-1},|t|^{q-1})$;
\end{itemize}
with $4<\a\le p<6<q$ and $C_1,$ $C_2,$ positive constants. We shall prove the following result

    \begin{theorem}\label{main}
        Assume that $f$ satisfies the above hypotheses, then there exists a couple $(u_0,\phi_0)\in \D\times\D$ which is a weak solution of \eqref{KGM}.
    \end{theorem}

Under the assumptions of Theorem \ref{main}, standard arguments (again Lemma \ref{main-1}) yield the existence of $(u,\phi)\in \H\times \D$ weak solutions of \eqref{main eq} in the case $\o<m$.
The limit case $\o=m$ is trickier.

Even if the claim of Theorem \ref{main} is analogous to the quoted existence results (e.g. Theorem \ref{th}) and the meaning of weak solution is the same, the approach in the proof is completely different. More precisely in the zero mass case, there exists no functional ${\cal S}$ defined on $\D\times\D$ such that its critical points are weak solutions of \eqref{KGM}.

As above we could consider a functional ${\cal S}:\H\times\D \to \R$ whose critical points are finite energy weak solutions.
For every $u\in\H$ we can find $\phi_u\in\D$ solution of \eqref{seconda}, then we could consider the reduced functional $I(u)={\cal S}(u,\phi_u)$. The reduced functional $I$ has the form
\[
I(u)={\cal S}(u,\phi_u)=\frac12 \irt |\n u|^2  +e\o \phi_u u^2
-\irt f(u).
\]
For such a functional the mountain pass geometry in $\H$ is not immediately available.

The solution $(u_0,\phi_0)\in \D\times\D$ will be found as limit of solutions of approximating problems
\begin{equation} \label{KGMeps}
\left\{
\begin{array}{ll}
-\Delta u +(\eps + 2e\o\phi-e^2\phi^2 )u-f'(u)=0 & \hbox{in }\RT,
\\
\Delta \phi=e(e\phi-\o) u^2 & \hbox{in }\RT,
\end{array}
\right.
\end{equation}
For every $\eps>0$, Lemma \ref{main-1} yields a solution $(u_\eps,\phi_\eps)\in \H\times \D$.
The stronger assumptions on $f$ (subcritical at infinity, supercritical at zero) give rise to uniform estimate in $\D\times\D$ which allows to pass to the limit as $\eps\to 0$.

Before giving the proof of Theorems \ref{th} and \ref{main}, let us recall some other results concerning \eqref{main eq}. In \cite{C} there are existence and nonexistence results when $f$ has a critical growth at infinity. In \cite{AP} it is proved the existence of a ground state for \eqref{eq} (under the existence assumptions of \cite{DM}). Other recent papers (e.g. \cite{bf5} and \cite{Lo}) are concerned with the Klein-Gordon-Maxwell system with a completely different kind of nonlinearity, satisfying
\[
\frac{1}{2}m^2t^2-f(t)\geq 0.
\]
The solutions in this case are called ``nontopological solitons''.
In \cite{bf5} it is proved the existence of a nontrivial solution if the coupling constant $e$ is sufficiently small.
There are also some results for the system \eqref{main eq} in a bounded spatial domain \cite{dps} and \cite{dps1}. In this situation existence and nonexistence of nontrivial solutions depend on the boundary conditions, the boundary data, the kind of nonlinearity and the value of $e$. Lastly let us make mention of the review paper \cite{Fort}, which contains a large amount of references on this topic, including existence results for other classes of Klein-Gordon-Maxwell systems, obtained with more general \emph{ansatz}.

In the next Sections we shall prove respectively Theorems \ref{th} and \ref{main}. Appendix A contains the proof of a certain inequality, used in Section 2, which involves only elementary Calculus arguments.
\section{Proof of Theorem \ref{th}}

We need the following:
\begin{lemma}\label{le:phi}
For any $u\in \H$, there exists a unique $\phi=\phi_u\in \D$ which satisfies
\[
\Delta \phi=e(e\phi-\o) u^2 \quad \hbox{in }\RT.
\]
Moreover, the map $\Phi:u\in\H\mapsto\phi_u\in\D$ is continuously
differentiable, and on the set $\{x\in \RT\mid u(x)\neq 0\}$,
\begin{equation}\label{eq:phi}
0\le \phi_u\le\frac{\o}{e}.
\end{equation}
\end{lemma}

\begin{proof}
The proof can be found in \cite{BF2,DM2}.
\end{proof}

\begin{theorem}\label{theo:red}
The pair $(u,\phi)\in \H\times \D$ is a solution of \eqref{eq} if and only if $u$ is a critical point of
\[
I(u)={\cal S}(u,\phi_u)=\frac12 \irt |\n u|^2 + (m^2-\o^2) u^2 +e\o \phi_u u^2
-\frac{1}{p}\irt |u|^{p},
\]
and $\phi=\phi_u$.
\end{theorem}

For the sake of simplicity we set $\O=m^2-\o^2>0$.

With our assumptions, it is a hard task to find bounded Palais-Smale sequences of functional $I$, therefore we use an indirect method developed by Struwe \cite{S} and Jeanjean \cite{JJ}.
We look for the critical points of the functional
$I_\l\in C^1(\Hr,\R)$
\begin{equation*}
I_\l(u)=\frac12 \irt |\n u|^2 +\O u^2 +e\o \phi_u u^2
-\frac{\l}{p}\irt |u|^{p},
\end{equation*}
for $\l$ close to 1, where
\[
\Hr:=\{u\in \H\mid u \hbox{ is radially symmetric}\,\}.
\]

Set $\d<1$ a positive number (which we will estimate later),
$J=[\d,1]$ and
    \begin{equation*}
        \Gamma:=\{\gamma\in C([0,1],\Hr)\mid \gamma(0)=0,
        I_\l(\gamma(1))<0,\;\forall \l\in J\}.
    \end{equation*}
Using a slightly modified version of \cite[Theorem 1.1]{JJ}, it can
be proved the following
    \begin{lemma}\label{le:J}
        If $\G\neq\emptyset$ and for every $\l\in J$
            \begin{equation}\label{eq:cl}
                c_\l:=\inf_{\gamma\in\Gamma}\max_{t\in[0,1]}
                I_\l(\gamma(t)) > 0,
            \end{equation}
        then for almost every $\l\in J$ there is a sequence
        $(v_n^\l)_n\subset \Hr$ such that
            \begin{itemize}
                \item[(i)] $(v_n^\l)_n$ is bounded;
                \item[(ii)] $I_\l(v_n^\l)\to c_\l$;
                \item[(iii)] $I_\l'(v_n^\l)\to 0.$
            \end{itemize}
    \end{lemma}
In order to apply Theorem \ref{le:J}, we have just to verify that
$\Gamma\neq\emptyset$ and \eqref{eq:cl}.

\begin{lemma}\label{le:Gamma}
For any $\l\in J$, we have that $\Gamma\neq\emptyset$.
\end{lemma}

\begin{proof}
Let $u\in \Hr\setminus\{0\}$ and let $\t>0$.
Define $\g:[0,1]\to\Hr$ such that $\g(t)=t \t u$, for all $t\in [0,1]$.
By \eqref{eq:phi}, for any $\l\in J$, we have that
    \begin{equation*}
        I_\l(\g(1)) =I_\l (\t u)\le \frac{\t^2}{2} \irt |\n u|^2 +\O u^2 +\o^2 u^2
- \d\frac{\t^{p}}{p}\irt |u|^{p},
    \end{equation*}
and then certainly $\g\in\G$ for a suitable choice of $\t$.
\end{proof}

\begin{lemma}\label{le:cl}
For any $\l\in J$, we have that $c_\l>0$.
\end{lemma}

\begin{proof}
Observe that for any $u\in\Hr$ and $\l\in J$, by \eqref{eq:phi}, we have
\[
I_\l(u) \ge \frac 1 2 \irt |\n u|^2+\O u^2 -\frac{1}{p}\irt |u|^{p},
\]
and then, by Sobolev embeddings, we conclude that there exists
$\rho
>0$ such that for any $\l\in J$ and $u\in\Hr$ with $u\neq 0$ and
$\|u\|\le\rho,$ it results $I_\l(u)>0.$ In particular, for any
$\|u\|=\rho,$ we have $I_\l(u) \ge \tilde c >0.$ Now fix $\l\in J$
and $\g\in\G.$ Since $\g(0)=0\neq\g(1)$ and $I_\l(\g(1))\le 0$,
certainly $\|\g(1)\|
> \rho.$ By continuity, we deduce that there exists $t_\g\in
(0,1)$ such that $\|\g(t_\g)\|=\rho.$ Therefore, for any $\l\in
J,$
    \begin{equation*}
        c_\l\ge \inf_{\g\in\G} I_\l(\g(t_\g)) \ge \tilde c >0.
    \end{equation*}
\end{proof}

\begin{proofth}
Let $\l\in J$ for which there exists a bounded Palais-Smale sequence
$(v_\l^n)_n$ in $\Hr$ for functional $I_\l$ at level $c_\l$, namely
    \begin{align*}
        I_\l(v_\l^n)&\to c_\l;
        \\
        I_\l'(v_\l^n)&\to 0\;\hbox{in } (\Hr)'.
    \end{align*}

Up to a subsequence, we can suppose that there exists $v_\l\in\Hr$
such that
    \begin{equation}\label{eq:weak}
        v_\l^n\rightharpoonup v_\l\;\hbox{weakly in }\Hr
    \end{equation}
and
    \begin{equation*}
        v_\l^n(x)\to v_\l(x)\;\hbox{a.e. in }\RN.
    \end{equation*}
We make the following claims:
    \begin{align}
        I_\l'(v_\l)&=0\label{eq:claim1},
        \\
        v_\l&\neq 0\nonumber
        \\
        I_\l(v_\l)& \le c_\l.\label{eq:claim2}
    \end{align}

Claim \eqref{eq:claim1} follows immediately by \cite[Lemma 2.7]{AP}.

Suppose by contradiction that $v_\l=0$, then, since $v_n^\l\to v_\l(\equiv 0)$ in $L^{p}(\RT)$ and $I'_\l (v_n^\l)[v_n^\l]=o_n(1)\|v_n^\l\|$, we have
\begin{align*}
\irt |\n v_n^\l|^2+\O (v_n^\l)^2
&\le \irt |\n v_n^\l|^2 +\O (v_n^\l)^2 +2e\o \phi_{v_n^\l} (v_n^\l)^2 -e^2\phi_{v_n^\l}^2 (v_n^\l)^2
\\
&=\l\irt |v_n^\l|^{p}+o_n(1)\|v_n^\l\|=o_n(1).
\end{align*}
Hence $v_n^\l\to 0$ in $\H$ and we get a contradiction with \eqref{eq:cl}.
\medskip

We pass to prove \eqref{eq:claim2}. Since $v_n^\l\to v_\l$ in $L^{p}(\RT)$, by \eqref{eq:weak},
by the weak lower semicontinuity of the $\H-$norm and by Fatou Lemma, we get $I_\l(v_\l)\le c_\l.$

Now we are allowed to consider a suitable $\l_n\nearrow 1$ such
that, for any $n\ge 1$, there exists $v_n\in\Hr\setminus\{0\}$
satisfying
    \begin{align}
        (I_{\l_n})'(v_n)&=0\;\hbox{in }(\Hr)'\label{eq:sol},\\
        I_{\l_n}(v_n)&\le c_{\l_n}.\label{eq:mp}
    \end{align}
We want to prove that such a sequence is bounded.
\\

By \cite{DM2}, $v_n$ satisfies the Pohozaev equality
\begin{equation}\label{eq:Poho}
\irt |\n v_n|^2
+  3\O v_n^2
+  5e\o\phi_{v_n}v_n^2
- 2e^2\phi_{v_n}^2v_n^2
-\frac{6\l_n}{p}\irt |v_n|^{p}=0.
\end{equation}
Therefore, by \eqref{eq:sol}, \eqref{eq:mp} and \eqref{eq:Poho} we
have that the following system holds
\begin{equation*}
\left\{
\begin{array}{l}
\irt \frac 12|\n v_n|^2
+  \frac 12 \O v_n^2
+  \frac{e\o}{2}\phi_{v_n}v_n^2
-\frac{\l_n}{p} |v_n|^{p} \le c_{\l_n},
\\
\irt |\n v_n|^2
+  3\O v_n^2
+  5e\o\phi_{v_n}v_n^2
- 2e^2\phi_{v_n}^2v_n^2
-\frac{6\l_n}{p} |v_n|^{p}=0,
\\
\irt |\n v_n|^2
+  \O v_n^2
+ 2 e\o\phi_{v_n}v_n^2
- e^2\phi_{v_n}^2v_n^2
-\l_n |v_n|^{p}=0.
\end{array}
\right.
\end{equation*}

Subtracting by the first the second multiplied by $\a$ and the third multiplied by $(1-6\a)/p$, we get
\[
\frac{p-2\alpha p-2+12\alpha}{2p}\irt |\n v_n|^2
+\irt \left[C_{p,\a}\O
+B_{p,\a}e\o\phi_{v_n}
+A_{p,\a}e^2 \phi_{v_n}^2 \right]v_n^2
\le c_{\l_n},
\]
where
\begin{align*}
C_{p,\alpha}  & =\frac{\left(  p-2\right)  (1-6\alpha)}{2p},\\
B_{p,\alpha}  & =\frac{p-10\alpha p-4+24\alpha}{2p},\\
A_{p,\alpha}  & =\frac{1+2\alpha\left(  p-3\right)  }{p}.
\end{align*}
\\
It is easy to see that
\[
\frac{p-2\alpha p-2+12\alpha}{2p}>0,
\]
if and only if
\begin{equation*}
\a>\frac{2-p}{2(6-p)}.
\end{equation*}
In the Appendix (see Lemma \ref{le:para}) we will prove that there exists $\a\in\left(\frac{2-p}{2(6-p)},\frac 16\right)$ such that
\begin{equation*}
C_{p,\a}\O
+B_{p,\a}e\o\phi_{v_n}
+A_{p,\a}e^2 \phi_{v_n}^2 \ge 0,
\end{equation*}
then we can argue that
\begin{equation}\label{eq:bound}
\|\n v_n\|_2\le C \hbox{ for all }n\ge 1.
\end{equation}
Moreover, by \eqref{eq:sol}, we have
\begin{equation}\label{eq:ne}
\O \irt v_n^2
\\
\le \irt |\n v_n|^2 +\O v_n^2 +2e\o \phi_{v_n} v_n^2 -e^2\phi_{v_n}^2 v_n^2=\l_n\irt |v_n|^{p}.
\end{equation}
Since for all $\eps>0$ there exists $C_\eps>0$ such that $t^p\le C_\eps t^6 +\eps t^2$, for all $t\ge 0$, taking $\eps=\O/2$, by \eqref{eq:ne} we get
\begin{equation*}
\frac{\O}{2}\irt v_n^2 \le C_\eps \irt v_n^{6}.
\end{equation*}
Therefore, by the Sobolev embedding $\D\hookrightarrow
L^{6}(\RT)$ and \eqref{eq:bound} we deduce that
$(v_n)_n$ is bounded in $\H$.

Up to a subsequence, there exists $v_0\in\Hr$ such that
    \begin{equation*}
        v_n\rightharpoonup v_0\;\hbox{weakly in }\Hr.
    \end{equation*}
By \eqref{eq:sol}, we have that
\begin{equation*}
I'(v_n)=(I_{\l_n})'(v_n) + (\l_n -1) |v_n|^{p-2}v_n= (\l_n-1)|v_n|^{p-2}v_n
\end{equation*}
so $(v_n)_n$ is a Palais-Smale sequence for the functional
$I|_{H^1_r},$ since the sequence $(|v_n|^{p-2}v_n)_n$ is bounded in
$\big(\Hr\big)'.$
\\
By \cite[Lemma 2.7]{AP}, we have that $I'(v_0)=0.$

To conclude the proof, it remains to check that $v_0\neq 0.$
\\
By \eqref{eq:sol}, we have
\[
\irt |\n v_n|^2 +\O v_n^2
\le
\irt |\n v_n|^2 +\O v_n^2 +2e\o \phi_{v_n} v_n^2 -e^2\phi_{v_n}^2 v_n^2
\le\irt |v_n|^{p}
\]
and then, there exists $C>0$ such that $\|v_n\|_p\ge C$. Since $v_n \to v_0$ in $L^p(\RT)$, we conclude.
\end{proofth}

\section{Proof of Theorem \ref{main}}

The following lemma generalizes the existence result of \cite{BF2}.
\begin{lemma}\label{main-1}
        Let $f$ satisfy the following hypotheses:
        \begin{itemize}
\item[({\bf f1})] $f\in C^1(\R,\R)$;
\item[({\bf f2})] $\exists \a>4$ such that $\forall t \in\R\setminus\{0\}:\;\a f(t)\le f'(t)t$;
\item[({\bf f5})] $f'(t)=o(|t|)$ as $t \to 0$;
\item[({\bf f6})] $\exists C_1,C_2\ge 0$ and $p<6$ such that $\forall t\in\R:$ $|f'(t)|\le C_1 + C_2 |t|^{p-1}$.
\end{itemize}
        Assume that $0< \o < m$. Then \eqref{main eq} admits a nontrivial weak solution $(u,\phi)\in \H\times \D$.
    \end{lemma}

We simply give an outline of the proof.

\begin{itemize}
\item Using the same reduction argument (Lemma \ref{le:phi} and Theorem  \ref{theo:red}) applied to \eqref{main eq},
it is immediately seen that that $(u,\phi)\in H^1(\RT)\times \D$ is a
solution of \eqref{main eq} if and only if $u\in\H$ is a critical point
of
\begin{equation*}
I(u)= \frac 12 \irt |\n u|^2 + (m^2-\o^2) u^2 + e\o \phi_u u^2 -\irt f(u),
\end{equation*}
and $\phi=\phi_{u}$.

\item The functional $I$ satisfies the Palais-Smale condition in $\Hr$.

\item The functional $I$ shows the Mountain Pass geometry.

\end{itemize}

\begin{remark}
If $f$ is odd, just like in \cite{BF2}, the $\Z_{2}$-Mountain Pass Theorem \cite{AR} yields infinitely many solutions.
\end{remark}

Now we can prove Theorem \ref{main}.

As we said in the Introduction, for every $\eps>0$, we consider the approximating problem \eqref{KGMeps}.
The above Lemma gives the solution $(u_\eps,\phi_\eps)\in H^1(\RT)\times \D$.
More precisely  they are mountain pass type solutions and they are radially symmetric, in the sense that $u_\eps\in\Hr$ is a critical point of
\begin{equation*}
I_\eps(u)= \frac 12 \irt |\n u|^2 + \eps u^2 + e\o \phi_u u^2 -\irt f(u),
\end{equation*}
at the level
    \begin{equation*}
        c_\eps=\inf_{g\in \G_\eps} \max_{\t\in [0,1]} I_\eps(g(\t)),
    \end{equation*}
where
\begin{equation*}
\G_\eps=\left\{g\in C\big([0,1],\H\big) \mid g(0)=0,\;I_\eps(g(1))\le
0, \;g(1)\neq 0\right\}.
\end{equation*}
Moreover, $u_\eps$ belongs to the Nehari manifold of $I_\eps$:
\[
\Ne_\eps=\left\{u\in\H\setminus\{0\}\mid
\irt |\n u|^2+\eps u^2 +2e\o\phi_u u^2
-e^2\phi_u^2 u^2=\irt f'(u)u\right\}.
\]
In the sequel, we will refer to those approximating solutions as
{\it $\eps-$solutions}.

\begin{lemma}\label{le:PM}
There exists $C>0$ such that $c_\eps<C$, for any $0<\eps<1$.
\end{lemma}

\begin{proof}
Let $0<\eps<1$ and $g\in\G_\eps$, we have
\[
c_\eps\le \max_{\t\in [0,1]} I_\eps(g(\t))=I_\eps(g(\t_0))\le I_1(g(\t_0)).
\]
\end{proof}

\begin{lemma}\label{le:>C}
There exists $C>0$ such that $\|u_\eps\|_{{\cal D}^{1,2}}\ge C$, for any $\eps>0$. Moreover, for any $\eps>0$,
\begin{equation}\label{eq:f'}
\irt f'(u_\eps)u_\eps \ge C.
\end{equation}
\end{lemma}

\begin{proof}
Since $u_\eps$ is solution of \eqref{KGMeps}, using \eqref{eq:phi}, we have
\begin{align*}
\irt |\n u_\eps|^2&\le \irt |\n u_\eps|^2+\eps u_\eps^2 + 2e\o\phi_{u_\eps} u_\eps^2
-e^2\phi_{u_\eps}^2 u^2_\eps=\irt f'(u_\eps)u_\eps\\
&\le \irt |u_\eps|^6\le C \left(\irt |\n u_\eps|^2\right)^3
\end{align*}
and so we get the conclusion.
\end{proof}
We need a uniform boundedness estimate on the family of the
$\eps-$solutions, letting $\eps$ go to zero.
\\
Actually, we have the following result
\begin{lemma}
There exists a positive constant $C$ which is a uniform upper bound for
the family $(u_\eps,\phi_{u_\eps})_{\eps>0}$ in the
$\D\times\D-$norm.
\end{lemma}

\begin{proof}
We have
\begin{align*}
\frac 12 \irt |\n u_\eps|^2 + \eps u_\eps^2
+ e\o \phi_{u_\eps} u^2_\eps -\irt f(u_\eps)&=c_\eps,
\\
\irt |\n u_\eps|^2+\eps u_\eps^2
+2e\o \phi_{u_\eps} u^2_\eps
-e^2 \phi_{u_\eps}^2 u_\eps^2
-\irt f'(u_\eps)u_\eps &=0.
\end{align*}
By Lemma \ref{le:PM} and ({\bf f2}) we deduce that
\begin{equation} \label{eq:bdd1}
\left(\frac \a 2 -1\right)\irt |\n u_\eps|^2 + \eps u_\eps^2
+\left(\frac \a 2-2\right) \irt e\o\phi_{u_\eps} u^2_\eps \le C,
\end{equation}
while, by the second equation of \eqref{KGMeps}, we have
\begin{equation}\label{eq:bdd2}
\irt |\n \phi_{u_\eps}|^2 +e^2\phi_{\eps}^2 u_\eps^2= \irt e\o\phi_{u_\eps}{u_\eps}^2.
\end{equation}
Combining together \eqref{eq:bdd1} and \eqref{eq:bdd2}, we infer that $(u_\eps,\phi_{u_\eps})_{\eps>0}$ is bounded in the
$\D\times\D-$norm.
\end{proof}

Now we deduce that, for any $\eps_n\to 0,$ there exist a
subsequence of $(u_{\eps_n},\phi_{u_{\eps_n}})_n$ (which we
relabel in the same way), and $(u_0,\phi_0)\in\D\times\D$ such
that
    \begin{align*}
        u_{\eps_n}&\rightharpoonup u_0, \quad\hbox{ in } \D,\\
        \phi_{u_{\eps_n}}&\rightharpoonup \phi_0, \quad\hbox{ in } \D.
    \end{align*}

We want to show that, if the sequence $(u_{\eps_n},\phi_{u_{\eps_n}})_n$ concentrates, then $(u_0,\phi_0)$ is a
weak nontrivial solution of \eqref{KGM}. From now on, we use $u_n$ and $\phi_n$ in the place of
$u_{\eps_n}$ and $\phi_{u_{\eps_n}}.$

\

%

Now we can prove the existence result in the limit case.

\begin{proofmain}
By \cite[Lemma 13]{BF} and \cite[Proposition 24]{P}, and by \eqref{eq:f'}, we have that
\[
\irt f'(u_0)u_0 =\lim_n \irt f'(u_n)u_n \ge C>0,
\]
and so $u_0\neq 0.$

Let us show that $(u_0,\phi_0)$ is a weak solution of \eqref{KGM}, namely
\begin{align*}
\irt \n u_0 \cdot \n \psi + 2 e\o\phi_0 u_0 \psi -e^2\phi_0^2 u_0 \psi &= \irt f'(u_0)\psi,
\\
\irt \n \phi_0 \cdot \n \psi +e^2\phi_0 u_0^2 \psi&=\irt e\o u_0^2 \psi ,
\end{align*}
for any $\psi$ test function.

Since, for any $n\ge 1$,  $(u_n,\phi_n)$ is a solution of \eqref{KGMeps}, we have
\begin{align*}
\irt \n u_n \cdot \n \psi +\eps_n u_n \psi + 2 e\o\phi_n u_n \psi -e^2\phi_n^2 u_n \psi&= \irt f'(u_n)\psi, 
\\
\irt \n \phi_n \cdot \n \psi+e^2\phi_n u_n^2 \psi&=\irt e\o u_n^2 \psi . 
\end{align*}
Let us prove that
\begin{equation}\label{eq:psi1}
\irt \phi_n u_n \psi \to \irt \phi_0 u_0 \psi.
\end{equation}
Indeed, denoting with $K=Supp(\psi)$, we observe that
\begin{align*}
\left| \irt \phi_n u_n \psi - \phi_0 u_0 \psi\right|
&\le \irt |\phi_n u_n \psi - \phi_n u_0 \psi|
+\irt |\phi_n u_0 \psi - \phi_0 u_0 \psi|
\\
& \le \irt |\phi_n||u_n-u_0| |\psi| +\irt |\phi_n-\phi_0||u_0||\psi|
\\
& \le \left( \irt |\phi_n|^6 \right)^{\frac 16}
\left( \int_{K} |u_n-u_0|^\frac65 \right)^{\frac 56}\sup |\psi|
\\
&\quad+\left( \int_{K} |\phi_n-\phi_0|^\frac65 \right)^{\frac 56}
\left( \irt |u_0|^6 \right)^{\frac 16}\sup |\psi|,
\end{align*}
and so we get \eqref{eq:psi1}, since $u_n\rightharpoonup u_0$ and $\phi_n\rightharpoonup\phi_0$ in $H^1(K)$.
\\
Let us prove that
\begin{equation}\label{eq:psi2}
\irt \phi_n^2 u_n \psi \to \irt \phi_0^2 u_0 \psi.
\end{equation}
Indeed, we have
\begin{align*}
\left|\irt \phi_n^2 u_n \psi  - \phi^2_0 u_0 \psi\right|&
\le \irt \phi^2_n |u_n-u_0| |\psi |
+ \irt |\phi^2_n-\phi^2_0| |u_0| |\psi|
\\
&\le \left( \irt |\phi_n|^6 \right)^{\frac 16} \left(\int_K |u_n-u_0|^{\frac 3 2}\right)^{\frac 2 3} \sup |\psi|
\\
&\quad+ \left(\int_K |\phi_n^2-\phi_0^2|^{\frac 65}\right)^{\frac 56} \left( \irt |u_0|^6 \right)^{\frac 16} \sup |\psi|
\\
&= o_n(1).
\end{align*}
\\
Therefore, by \eqref{eq:psi1} and \eqref{eq:psi2} and since $\psi$ has compact support, we have
\begin{equation*}
\begin{array}{ccccccccc}
\underbrace{\irt \n u_n \cdot \n\psi}_{\downarrow}
&+
&\!\!\underbrace{\irt \eps_n u_n\psi}_{\downarrow}\!\!
&+
&\underbrace{\irt 2e\o \phi_n u_n \psi}_{\downarrow}
&-
&\!\!\!\underbrace{\irt e^2\phi_n^2 u_n \psi}_{\downarrow}
&=
&\!\!\underbrace{\irt f'(u_n)\psi}_{\downarrow},
\\
\displaystyle \irt \n u_0 \cdot \n\psi
&+
&0
&+
&\displaystyle \irt 2e\o\phi_0 u_0\psi
&-
&\displaystyle \irt e^2\phi^2_0 u_0\psi
&=
&\!\!\displaystyle \irt f'(u_0)\psi.
\end{array}
\end{equation*}
Analogously, we have
\begin{equation*}
\begin{array}{ccccc}
\underbrace{\irt \n \phi_n \cdot \n\psi}_{\downarrow}
&+
&\!\!\!\underbrace{\irt e^2\phi_n u_n^2 \psi}_{\downarrow}
&=
&\!\!\underbrace{\irt u_n^2\psi}_{\downarrow},
\\
\displaystyle \irt \n \phi_0 \cdot \n\psi
&+
&\displaystyle \irt e^2\phi_0 u_0^2\psi
&=
&\!\!\displaystyle \irt u_0^2\psi.
\end{array}
\end{equation*}
In particular, by this last identity, we infer that $\phi_0\neq0$ and we conclude.
\end{proofmain}

\appendix
\section{Appendix}

\begin{lemma}\label{le:para}
Let $p\in(2,4)$ and $\omega\in(0,g(p)\,m)$. Then there exists
$\alpha\in I_{p}=\left(  \frac{2-p}{2(6-p)},\frac{1}{6}\right)  $ such that
\begin{equation*}
A_{p,\a}e^2 \phi_{v_n}^2
+B_{p,\a}e\o\phi_{v_n}
+ C_{p,\a}\O\ge 0,
\end{equation*}
where
\begin{align*}
A_{p,\alpha}  & =\frac{1+2\alpha\left(  p-3\right)  }{p},\\
B_{p,\alpha}  & =\frac{p-10\alpha p-4+24\alpha}{2p},\\
C_{p,\alpha}  & =\frac{\left(  p-2\right)  (1-6\alpha)}{2p}.
\end{align*}
\end{lemma}

\begin{proof}
Keeping in mind \eqref{eq:phi}, we have to show that
\begin{equation}\label{eq:f}
f(t)=A_{p,\a} t^2+B_{p,\a} \o t +C_{p,\a}\O\ge 0, \qquad \hbox{for any }t\in [0,\o].
\end{equation}
First we notice that for any $\alpha\in I_{p}$%
\begin{align*}
A_{p,\alpha}  & >0\\
C_{p,\alpha}  & >0
\end{align*}
\\
Now we have to distinguish two cases: $p\in(3,4)$ and $p\in(2,3]$.
\\
In the first one, if $\a=\frac{4-p}{24-10p}\in I_p$, we have $B_{p,\a}=0$ and so we have proved \eqref{eq:f}.
\\
Let now consider the case $p\in(2, 3]$. Since $f$ reaches its minimum in $-\frac{B_{p,\a}\o}{2A_{p,\a}}$ and it belongs to $[0,\o]$, $f$ is non-negative in $[0,\o]$ if and only if
\[
f\left(-\frac{B_{p,\a}\o}{2A_{p,\a}}\right)\ge 0,
\]
and, with straightforward calculations and using the fact that $A_{p,\a}+B_{p,\a}=C_{p,\a}$, this is equivalent to say that
\begin{equation}\label{eq:AC}
\frac{m^2}{\o^2}\ge \frac{(A_{p,\a}+C_{p,\a})^2}{4A_{p,\a}C_{p,\a}}.
\end{equation}
We set
\[
K_p(\a):= \frac{(A_{p,\a}+C_{p,\a})^2}{4A_{p,\a}C_{p,\a}}=\frac{p^2}{8(p-2)}\cdot\frac{(1-2\a)^2}{1-6\a}\cdot\frac{1}{1+2\a(p-3)}
\]
and we shall prove that
\begin{equation}\label{eq:Kp}
\inf_{\a\in I_p}K_p(\a)=\frac{1}{(p-2)(4-p)},
\end{equation}
and then we could conclude. Indeed, if $\o\in (0,g(p)m)$, then by \eqref{eq:Kp}
\[
\frac{m^2}{\o^2}>\inf_{\a\in I_p}K_p(\a),
\]
and so there exists $\a\in I_p$ such that
\[
\frac{m^2}{\o^2}\ge K_p(\a),
\]
by which we deduce \eqref{eq:AC}.
\\
Let us now prove \eqref{eq:Kp}.
\\
Let us consider the case $p=3$: in such situation
\[
K_3(\a)=\frac{9}{8}\cdot\frac{(1-2\a)^2}{1-6\a} \quad \hbox{ and }\quad I_3=\left(-\frac 16,\frac 16\right).
\]
Since the function $K_3$ is increasing in $I_3$, we have
\begin{equation*}
\inf_{\a\in I_3}K_3=K_3\left(-\frac 16 \right)=1
\end{equation*}
and so we have proved \eqref{eq:Kp}.
\\
Now, let us consider the case $p\in(2,3)$. We write $K_p(\a)= \frac{p^2}{8(p-2)}\cdot H_1(\a) \cdot H_2(\a)$ where
\[
H_1(\a):=\frac{(1-2\a)^2}{1-6\a}, \qquad H_2(\a):=\frac{1}{1+2\a(p-3)}.
\]
Since for $i=1,2$, $H_i$ is a positive and increasing function in $I_p$, we have
\[
\inf_{\a\in I_p}K_p=K_p\left(\frac{2-p}{2(6-p)}\right)=\frac{1}{(p-2)(4-p)},
\]
and so we obtain \eqref{eq:Kp}.
\end{proof}

\end{document}